\documentclass[11pt]{amsart}
\usepackage{amsthm,amsmath,amssymb,mathrsfs,yhmath}
\usepackage{xcolor}
\usepackage{enumerate}
\usepackage{verbatim}
\usepackage[active]{srcltx}  

\newtheorem{thm}{Theorem}

\newtheorem{lem}[thm]{Lemma}

\theoremstyle{definition}

\theoremstyle{remark}

\newtheorem*{rem}{Remark}

\newcommand{\ii}{\mathbf{i} }

\newcommand{\pac}{\mathcal{P}}
\newcommand{\hau}{\mathcal{H}}
\newcommand{\rr}{\mathbb{R}}
\newcommand{\nn}{\mathbb{N}}

\begin{document}
\title{On packing measures and a theorem of Besicovitch}

\author[Ignacio Garcia]{Ignacio Garcia $^\dag$}
\address{Departamento de Matem\'atica, Facultad de Ciencias Exactas y Naturales,
Universidad Nacional de Mar del Plata, Argentina}
\email{nacholma@gmail.com}
\author[Pablo Shmerkin]{Pablo Shmerkin $^\ddag$}
\address{Department of Mathematics, Faculty of Engineering and Physical Sciences,
University of Surrey, Guilford, GU2 7XH, United Kingdom}
\email{p.shmerkin@surrey.ac.uk}
\thanks{$\dag$ Partially supported by CAI+D2009 $\text{N}^\circ 62$-$310$
(Universidad Nacional del Litoral) and $E449$ (UNMDP)}
\thanks{$\ddag$ Partially supported by a Leverhulme Early Career Fellowship and by a Cesar Milstein Grant}

\begin{abstract} Let $\hau^h$ be the $h$-dimensional Hausdorff measure on  $\rr^d$.  Besicovitch showed that if a set $E$ is null for $\hau^h$, then it is null for $\hau^g$, for some dimension $g$ smaller than $h$. We prove that this is not true for packing measures. Moreover, we consider the corresponding questions for sets of non-$\sigma$-finite packing measure, and for pre-packing measure instead of packing measure.
\end{abstract}

\maketitle

\section{Introduction and statements of results}

\subsection{Introduction} We begin by recalling some definitions. A continuous nondecreasing
function $h:(0,\infty)\to(0,\infty)$ that verifies $h(t)\to0$ as
$t\to0$ is called a \emph{dimension function}, and the set of these
functions is denoted by $\mathcal D$.

We will be concerned with packing measures. Recall that, given $\delta>0$, a $\delta$-packing of a subset
$E\subset\rr^d$ is a collection of disjoint open balls centered at
$E$ with diameter less than $\delta$. Given $h\in\mathcal D$, the
\emph{$h$-dimensional packing premeasure} $P_0^h$ is given by
\begin{align*}
P_\delta^h(E)&=\sup\left\{\sum_{i} h(|B_i|): \{B_i\}\ \text{is} \
a\ \delta\text{-packing of } E \right\},\\
P_0^h(E) &= \inf_{\delta>0} P_\delta^h(E) = \lim_{\delta\downarrow 0} P_\delta^h(E),
\end{align*}
where $|B|$ denotes the diameter of the ball $B$. The function $E\to P_0^h(E)$ is monotone and finitely subadditive, but fails to be countably subadditive, even on nice sets. The \emph{ $h$-dimensional packing measure} of $E$ (or $h$-packing measure of $E$) is defined by
\begin{equation*}
\pac^h(E)=\inf\left\{\sum_{j=1}^\infty P_0^h(E_j): E\subset\bigcup_{j=1}^\infty E_j \right\}.
\end{equation*}
It is well known that $\pac^h$ is an outer measure on $\rr^d$ (i.e. a monotone, countably subadditive set function which vanishes on the empty set), and is a measure (countably additive on disjoint collections) on the class of analytic (or Suslin) sets.

For a given dimension function $h$, the packing measure $\pac^h$ is in some sense dual to the Hausdorff measure $\hau^h$ (for whose definition the reader is referred to \cite[Section 2.5]{Fal90}; in this note we do not use Hausdorff measures other than for motivation). Just as Hausdorff measures, packing measures are used to provide fine information on the size of fractal sets. For many random sets, especially related to Brownian motion, packing measures (rather than Hausdorff measures) provide the ``right'' concept to measure the size of the set, see e.g. \cite{Duquesne09} and \cite{MortersShieh09} for two deep recent examples. Packing measures have also been recently applied to the study Cantor sets defined in terms of their gaps \cite{CHM10} and their rearrangements \cite{HMZ11}.

Given $g,h\in\mathcal D$, we say that $g$ is a smaller dimension
than $h$, denoted by $g\prec h$, whenever
\[
\lim_{t\to0}\frac{h(t)}{g(t)}=0.
\]
This relation defines a partial order on $\mathcal D$. A basic property of packing measures states that if $\pac^h|_E$
is a $\sigma$-finite nontrivial measure space, then $\pac^f(E)=0$ for any $f\in\mathcal D$ such
that $h\prec f$, and $\pac^g|_E$ is non-$\sigma$-finite for any
$g\in\mathcal D$ such that $g\prec h$ (here $\pac^h|_E$ is the
restriction of the measure $\pac^h$ to the set $E$).

 It follows from the above that the poset $\{ f\in\mathcal{D}: \pac^f(E)=0\}$ has no maximal elements, and a natural question is whether it may have minimal elements; likewise, one may ask whether $\{ f\in\mathcal{D}:\pac^f|_E \text{ is non-}\sigma\text{-finite} \}$ has maximal elements. In the context of Hausdorff measures, Besicovitch proved the the answer is always negative, at least if the set $E$ is analytic (see also \cite[Theorem 42]{Rog}). More precisely, we have

\begin{thm}[\cite{Besi56b}]\label{besi}
Let $E\subset\rr^d$ and $h\in\mathcal D$.
\begin{enumerate}[(a)]
\item If $\hau^h(E)=0$ then there exists
$g\in\mathcal D$ such that $g\prec h$ and also $\hau^g(E)=0$.
\item If $E$ is analytic and $\hau^h|_E$ is non-$\sigma$-finite, then there exists $f\in\mathcal D$ such that
$h\prec f$ and also $\hau^f|_E$ is non-$\sigma$-finite.
\end{enumerate}
\end{thm}

In this note we discuss Besicovitch's Theorem but in the setting of packing measures.

\subsection{Results}
Our first result is that the first part of Theorem \ref{besi} fails for packing measures. In fact, let $\mathcal D_d\subset
\mathcal D$ be the set of the `at most' $d$-dimensional doubling
dimension functions, that is, functions $h\in\mathcal D$ that
verifies
\begin{equation}\label{dimension}
\frac{t^d}{h(t)}< c_0, \ \forall t>0, \quad \text{for some } c_0=c_0(h)
\end{equation}
and
\begin{equation}\label{doubling}
 h(2t)\le c_1h(t), \
\forall t>0, \quad \text{for some } c_1=c_1(h).
\end{equation}
For example,  $\mathcal D_d$ contains the power functions $t^s$ for
$0<s\le d$ as well as $t^s \varphi(t)$ when $\varphi(t)$ is slowly varying.

\begin{thm}\label{paczero} Given $h\in\mathcal D_d$, there is a Borel (in fact, $G_\delta$) set $E\subset\rr^d$ such that
$\mathcal P^h(E)=0$ but $\mathcal P^g|_E$ is non-$\sigma$-finite for
any $g\in\mathcal D$ such that
\begin{equation} \label{eq:liminf_zero}
\liminf_{t\to0}\frac{h(t)}{g(t)}=0.
\end{equation}
In particular, $\mathcal P^g|_E$ is not $\sigma$-finite if $g\prec
h$.
\end{thm}






Now we turn to the second part of Besicovitch's Theorem.
Recall that an {\em analytic set} is the continuous image of $\nn^\nn$, where  $\nn^\nn$ is endowed with the product topology.

\begin{thm}\label{pacinfty}
Let $A\subset\rr^d$ be an analytic set of non-$\sigma$-finite $\mathcal P^h$
measure. Then there exists $g\in\mathcal D$ such that $h\prec g$ and $A$ has non-$\sigma$-finite $\mathcal P^g$ measure.
\end{thm}

Finally, for prepacking measures we have the following {\em general} result.

\begin{thm}\label{prepac}
Let $A\subset\rr^d$ and $h\in\mathcal D$.
\begin{enumerate}[(a)]
\item \label{it:prepac-down} If $P_0^h(A)=0$, then exists $g\in\mathcal D$ such that
$g\prec h$ and $P_0^g(A)=0$.
\item \label{it:prepac-up} If $P_0^h(A)=+\infty$, then
exists $g\in\mathcal D$ such that $h\prec g$ and $P_0^g(A)=+\infty$.
\end{enumerate}
\end{thm}

In Section \ref{sec:proofs} we give the proofs of the results, and in Section \ref{sec:remarks} we conclude with remarks and a question.

\section{Proofs} \label{sec:proofs}

\subsection{Proof of Theorem \ref{paczero}}
In this note, $B_r(x)$ always denote the open ball with center $x$ and radius $r$.
We need a preliminary lemma, which is an immediate consequence of the density theorem for packing measure (see \cite{TT}).

\begin{lem}\label{TT}
Let $A\subset\rr^d$ and $h\in\mathcal D_d$. If there is a finite
Borel measure $\mu$ on $\rr^d$ that verifies
\begin{equation*}
C^{-1} h(r)\le\mu(B_r(x))\le Ch(r), \ \forall r<1, \ \forall x\in A,
\end{equation*}
for some positive and finite constant $C$, then
$0<\pac^h(A)<+\infty$.
\end{lem}



\begin{proof}[Proof of Theorem \ref{paczero}]
It is well known that a dense $G_\delta$ subset of $\rr^d$ always has full packing dimension, even though it may have zero Lebesgue measure and even zero Hausdorff dimension. Our idea is to replace $\rr^d$ by an appropriate set $K$ of positive, finite $h$-dimensional packing measure, construct such a dense $G_\delta$ set of zero $\pac^h$ measure, and show that dense $G_\delta$ sets (relative to $K$) have non-$\sigma$-finite measure for any $g$ satisfying \eqref{eq:liminf_zero}.

\textbf{First step. Construction of the Cantor set $K$}.
Without loss of generality we assume that $c_0=1$, where $c_0$ is the constant in (\ref{dimension}). Let
$(a_n)_{n\ge1}$ be defined by $h(a_n)=2^{-dn}$. From
(\ref{dimension}) and (\ref{doubling}) we have that $2^na_n<1$ and
$2a_{n+1}< a_n$, whence we can construct a $2^d$-corner Cantor set
$K$ in the unit cube as follows. Let $K_0=[0,1]^d$ and define
$K_n=\bigcup_{j=1}^{2^d} Q_j^n$, for $n>0$, as the union of the
$2^{dn}$ {\em basic cubes} $Q_j^n$ of side $a_n$ contained in
$K_{n-1}$ that have a common vertex with a cube of $K_{n-1}$. Then
$K=\bigcap_{n\ge0}K_n$.

Let $\mu$ be the uniform Cantor measure on $K$, which is constructed
by repeated subdivision setting
$$
\mu(Q_j^n)=2^{-dn}, \ 1\le j\le2^{dn}, \ n\ge0.
$$
(See for example \cite[Chapter 1]{Fal90}.) Now fix $r<1$ and $x\in K$. Let $n$
be the greatest integer such that $B_r(x)$ contains a basic cube of
$K_n$ but none of $K_{n-1}$. Then $a_n\le 2r$ and
$r<\sqrt{d}\,a_{n-1}$, whence $$Ch(r)\le h(a_n)\le c_1h(r),$$ where
$C>0$ is independent of $r$. Moreover, $B_r(x)$ intersects at most
$2^{d+1}$ basic cubes of $K_n$. Then, by the definition of $a_n$, we
have
\begin{equation*}
h(a_n)\le \mu(B_r(x))\le 2^{d+1} h(a_n),
\end{equation*}
and $0<\pac^h(K)<+\infty$ by Lemma \ref{TT}. Since $\pac^h$ is invariant under isometries of $\rr^d$, it assigns the same mass to all basic cubes of $K_n$ for all $n\ge 0$. Invoking also the Borel regularity of these measures, it follows that $\mu=(1/\pac^h(K))\pac^h|_K$. In particular, $\mu$ and $\pac^h|_K$ have the same null sets.

\textbf{Second step. Construction of the zero measure, dense $G_\delta$ set $E$}. For $n\ge0$, let $V_n$ be the set of
all the vertices of basic cubes of $K_n$. Let
$$U_k=\bigcup_{n\ge0}\bigcup_{v\in V_n} K\cap B_{a_{2n+k}}(v)\setminus\{v\}.$$
Each $U_k$ is dense in $K$ and open with respect to the relative
topology. We define $E=\bigcap_{k\ge1} U_k$. The Baire category theorem
implies that $E$ is a dense subset of $K$. Moreover,
for each $k\ge1$,
\begin{equation*}
\mu(U_k)\le2^{d+1}\sum_{n\ge0}\sum_{v\in
V_n}h(a_{2n+k})=\frac{2^{2d+1}}{2^d-1}2^{-dk},
\end{equation*}
which implies $\pac^h(E)=0$.

\textbf{Third step. Conclusion of the proof}. Now let $g\in\mathcal D$ satisfy \eqref{eq:liminf_zero}.
Let $(r_k)$ be a sequence decreasing to $0$ such that
$h(r_k)/g(r_k)\to 0$ and let $(n_k)$ be such that $a_{n_k+1}<r_k\le
a_{n_k}$.

We first claim that $P_0^g(U\cap K)=+\infty$ for any open set
$U\subset\rr^d$ with $U\cap K\neq\varnothing$. Indeed, it is enough
to show that $P_0^g(Q\cap K)=+\infty$ for each basic cube $Q$ but,
since $P_0^g$ is invariant under Euclidian isometries, the symmetry
of $K$ implies that $P_0^g(Q\cap K)$ is constant over all basic
cubes in $K_n$ for each $n$ whence, by finite subadditivity of
$P_0^g$, it is enough to show that $P_0^g(K)=+\infty$. Let
$\delta>0$ and pick $k$ large enough that $r_k<\delta$. For each
basic cube, let $v(Q)=\min\{v\in Q\}$, where the minimum is taken
with respect to the lexicographical order (that is, $v(Q)$ is the
{\em down-left vertex} of $Q$). Then $\{ B_{r_k}(v(Q)): Q\text{ is a
basic cube in } K_{n_k-1}\}$ is a $\delta$-packing of $K$, and it
follows that
\begin{equation*}
P_\delta^g(K)\ge 2^{d (n_k-1)} g(r_{n_k}) = \left(2^{d(n_k-1)}h(r_{n_k})\right)\frac{g(r_{n_k})}{h(r_{n_k})}.
\end{equation*}
Since $n_k$ is arbitrarily large, this shows that
$P_\delta^g(K)=+\infty$ for all $\delta>0$, establishing the claim.

Now suppose $E$ has $\sigma$-finite $\pac^g$-measure. Then
$E=\bigcup_i E_i$, where $\pac^g(E_i)<+\infty$ for all $i$. In
particular, each $E_i$ can be written as $E_i =\bigcup E_{ij}$,
where $P_0^g(E_{ij})<+\infty$ for all $i,j$. Since the packing
pre-measure of a set and its closure are equal, we also have
$P_0^g(\overline{E}_{ij})<+\infty$ for all $i,j$. Hence, by the
claim, each $E_{ij}$ is nowhere dense. But this would imply that
$K=(K\setminus E)\cup E$ is the union of two meager sets,
contradicting the Baire category theorem. This contradiction
finishes the proof.
\end{proof}

\subsection{Proof of Theorem  \ref{pacinfty}}
The proof of the theorem uses the ideas of Haase (\cite{H86},
Theorem 2), where it is shown that an analytic set of
non-$\sigma$-finite $h$-packing measure contains a compact subset of
non-$\sigma$-finite $h$-packing measure. We provide full details for the reader's convenience.

We endow $\nn^\nn$ with the metric $\rho$ defined for $\ii=(i_1, i_2, \ldots)$ and $\mathbf{j}=(j_1, j_2, \ldots)$ in $\nn^\nn$ by $\rho(\ii,
\mathbf j)=1/k_0$, where $k_0$ is the smallest integer $k$ such that
$i_k\neq j_k$. This metric induces the product topology on $\nn^\nn$. Given $(i_1, \ldots, i_k)\in\nn^k$, the {\em cylinder in $\nn^\nn$ of level $k$ associated to $(i_1, \ldots, i_k)$} is  the clopen set
$G_{j_1,\ldots,j_k}=\{\ii\in\nn^\nn: i_1=j_1,\ldots, i_k=j_k\}$. For
each $k\ge1$, note that $\nn^\nn$ is the countable (disjoint) union of all the cylinders of level $k$. Also, each such cylinder has diameter $1/k$.

We say that a set $E\subset\rr^d$ has {\em locally non-$\sigma$-finite
$h$-packing measure} if $U\cap E$ has non-$\sigma$-finite
$h$-packing measure for each open set $U$ with $U\cap E\neq\emptyset$. We need a preliminary lemma. Let $\varphi:\nn^\nn\to A$ be a surjective continuous function, which exists because $A$ is analytic.

\begin{lem}\label{lemma-closed}
Given $k>0$ and a closed set $C\subset\nn^\nn$ whose image
$\varphi(C)$ has non-$\sigma$-finite $h$-packing measure, then there
is a closed subset $\tilde C\subset C$ such that $\varphi(\tilde C)$
has locally non-$\sigma$-finite $h$-packing measure and diam$(\tilde
C)\le1/k$.
\end{lem}
\begin{proof}
Note that for some $(j_1,\ldots, j_k)\in\nn^k$, the image of the closed set $D=C\cap G_{j_1,\ldots, j_k}$ under $\varphi$ has non-$\sigma$-finite measure. Also $D$ has diameter at most $1/k$.
Let $$\mathcal U=\bigl\{U:  U \text{ open, and } U\cap \varphi(D) \text{ is nonempty and has }
\sigma\text{-finite measure}\bigr\}.$$ By the Lindel\"{o}f property, there is a countable subfamily
$\{U_n\}$ of $\mathcal U$ for which $\bigcup_n
U_n=\bigcup_{U\in\mathcal U}U$. Then $F=\varphi(D)\setminus\bigcup_nU_n$
has locally non-$\sigma$-finite measure. Hence, the set $\tilde C:=\varphi^{-1}(F)\cap D=D\setminus\varphi^{-1}(\bigcup_nU_n)$ verifies
the statement of the lemma.
\end{proof}

\begin{proof}[Proof of Theorem  \ref{pacinfty}]
We begin with some notation. For $\ii=(i_1, i_2,\ldots)\in\nn^\nn$
let $\ii_n=(i_1, \ldots, i_n)$. Also $\ii_nj$ denotes the $(n+1)$-tuple $(i_1, \ldots, i_n, j)$.

Let $C_1\subset \nn^\nn$ be the set obtained applying Lemma
\ref{lemma-closed} to $C=\nn^\nn$ and $k=1$. Set $a_2=1$ and let
$\{B_{r_{1j}}(x_{1j})\}_{j\in \mathcal D}$ be a finite $a_2$-packing
of $\varphi(C_1)$ such that
$$\sum_{j\in \mathcal D}h(2r_{1j})>4^2.$$
Since $\mathcal D$ is finite, there exists $0<a_3<\min_{j\in\mathcal D}
2r_{1j}$ such that if
$y_j\in\overline B_{a_3}(x_{1j})$ for each $j\in\mathcal D$, then
the balls $B_{r_{1j}}(y_{j})$ are disjoint. Next, for $j\in\mathcal
D$, we define $C_{1j}\subset C_1$ applying Lemma \ref{lemma-closed}
to $\varphi^{-1}(\overline B_{a_3}(x_{1j}))\cap C_1$ and $k=2$.

Continuing inductively in this fashion, we construct the following items:
\begin{enumerate}
\item A sequence $(a_n)$ strictly decreasing to $0$.
\item \label{it:finite-level} A subset $\mathcal T\subset\nn^\nn$ such that
$\# \mathcal D(\ii,n)<+\infty$ for each $\ii\in \mathcal T$ and
$n>0$, where $$\mathcal D(\ii, n)=\{j\in\nn : \ii_nj=\mathbf{l} \text{ for  some } \mathbf{l}\in\mathcal T \}.$$
\item \label{it:closed-set}A family of closed subsets $C_{\ii_n}\in\nn^\nn$,
indexed by taking $\ii\in\mathcal T$ and $n>0$, that verifies
\begin{enumerate}
\item $C_{\ii_{n+1}}\subset C_{\ii_n}$;
\item diam$\: C_{\ii_n}\le1/n$;
\item $\varphi(C_{\ii_n})$ has locally non-$\sigma$-finite $h$-packing measure.
\end{enumerate}
\item \label{it:radii} For each $\ii\in\mathcal T$ and $n>0$, a
finite $a_{n+1}$-packing
$\bigl\{B_{r_{\ii_nj}}(x_{\ii_nj})\bigr\}_{j\in\mathcal D(\ii,n)}$
for the set $\varphi(C_{\ii_n})$ such that
\begin{enumerate}
\item \label{it:sum-h-over-level} $$\sum_{j\in\mathcal
D(\ii,n)}h(2r_{\ii_nj})>4^{n+1};$$
\item \label{it:radius} $a_{n+2}< 2r_{\ii_nj}<a_{n+1}$ for all $j\in\mathcal D(\ii,n)$;
\item  given $j\in\mathcal D(\ii,n)$ and $y_j\in\overline B_{a_{n+2}}(x_{\ii_nj})$,
the open balls $B_{r_{\ii_nj}}(y_j)$ are pairwise disjoint;
\item \label{it:contained-in-ball}$\varphi(C_{\ii_n})\subset \overline B_{a_{n+1}}(x_{\ii_n})$.
\end{enumerate}
\end{enumerate}

The set $K=\bigcap_n\bigcup_{\ii\in\mathcal T} C_{\ii_n}$ is compact
because it is closed and totally bounded (by \eqref{it:finite-level}
and \eqref{it:closed-set}). Hence $E=\varphi(K)$ is a compact subset
of $A$.

We define
\begin{equation*}
 g(t)=
        \left\{
  \begin{array}{ccc}
     h(t)/2^{n-\frac{t-a_n}{a_{n-1}-a_n}}, & t\in(a_n, a_{n-1}], \ n>1 \\
        h(t), & t>a_1
  \end{array}
    \right..
\end{equation*}
It is easy to check that $g$ is continuous and $h\prec g$.

Let $U$ be an open set such that $U\cap E\neq\emptyset$. Then there
exist $\ii\in\mathcal T$ and $n_0>0$ such that $\overline
B_{a_{n+1}}(x_{\ii_n})\subset U$ for all $n\ge n_0$, whence $E\cap
\overline B_{a_{n+1}}(x_{\ii_n})\neq \emptyset$ using that $K\cap
C_{\ii_n}\neq\emptyset$ and \eqref{it:contained-in-ball}. In
particular, for each $j\in\mathcal D(\ii, n)$ there exists $y_j\in
E\cap\overline B_{a_{n+2}}(x_{\ii_nj})$. By construction, the balls
$B_{r_{\ii_nj}}(y_j)$ are a packing of $E\cap U$, and by
\eqref{it:sum-h-over-level} and \eqref{it:radius} we have
\begin{equation*}\label{ultima}
P_{a_{n+1}}^g(U\cap E)\ge\sum_{j\in\mathcal
D(\ii,n)}g(2r_{\ii_nj})>2^{n},
\end{equation*}
whence $P_0^g(U\cap E)=+\infty$. By a Baire category argument (see \cite[Lemma 4]{H86}), we
conclude that the compact set $E$ has non-$\sigma$-finite
$g$-packing measure, which implies that the same holds for $A$.
\end{proof}

%

\subsection{Proof of Theorem \ref{prepac}}

\begin{proof}[Proof of Theorem \ref{prepac}\eqref{it:prepac-down}]
It is enough to find $f\in\mathcal D$ such that $f\prec h$ and
$P_0^f(A)<+\infty$, because in this situation any $g\in\mathcal D$
such that $f\prec g\prec h$ satisfies the conclusion of the theorem.

Let $(\delta_n)_{n\ge0}\searrow0$ be a decreasing sequence of
positive reals such that
\begin{enumerate}
\item $P_{\delta_n}^h(A)\le1/{2^{2n}}$, and
\item $2h(\delta_{n+1})<h(\delta_n)$.
\end{enumerate}
If $0<\delta<\delta_0$ and $\mathcal B$ is a $\delta$-packing of
$A$, set
\[\mathcal B_n=\{B\in\mathcal B:\delta_{n+1}\le|B|<\delta_n\}.\]
Then, by $(1)$
\begin{align*}
\sum_{n\ge0}\sum_{B\in\mathcal B_n}2^nh(|B|)\le \sum_{n\ge0}2^n
P_{\delta_n}^h(A)\le\sum_{n\ge0}2^{-n}=2,
\end{align*}
whence we define $f$ by
\begin{equation*}
         f(t)=
        \left\{
  \begin{array}{cc}
      \max\left(2^nh(t), 2^{n+1}h(\delta_{n+1})\right), & t\in[\delta_{n+1}, \delta_n)\\
     h(t), & t\ge\delta_0
  \end{array}
    \right..
\end{equation*}
It follows from $(2)$ that $f(\delta_n)=2^nh(\delta_n)$, hence $f$
is (left) continuous. It is also monotone nondecreasing and
$f(t)\to0$ as $t\to0$ since
\[f(\delta_n)=2^nh(\delta_n)\le 2^nP_{\delta_n}^h(A)\le1/2^n.\]
Moreover, $f\prec h$ because
\begin{align*}
\frac{h(t)}{f(t)}\le\frac{1}{2^n}\quad\text{for } t\in[\delta_{n+1}, \delta_n).
\end{align*}
Finally, for $t\in[\delta_{n+1},
\delta_n)$ we have $f(t)\le 2^{n+1}h(t)$, therefore if $\mathcal B$ is an arbitrary packing and $\mathcal B_n$
is as above, then
\begin{align*}
\sum_{n\ge0}\sum_{B\in\mathcal B_n}
f(|B|)=\sum_{n\ge0}\sum_{B\in\mathcal B_n} 2^{n+1}h(|B|)\le4,
\end{align*}
and the result follows.
\end{proof}

For the proof of Theorem \ref{prepac}\eqref{it:prepac-up} we use the following
lemma.

\begin{lem}\label{propopacinfinite}
Let $C\subset\rr^d$ and $h\in\mathcal D$. Then, $P_0^h(C)=+\infty$
if and only if there is a $1$-packing $\{B_j\}$ of $C$ which
verifies $\sum_j h(|B_j|)=+\infty$.
\end{lem}
\begin{proof}
We assume that $C$ is bounded. Otherwise the equivalence is trivial
from the definition of $P_0^h$.

Let $\delta>0$ and assume that there exists a $1$-packing as in the
statement. Then, there exists $J>0$ such that $|B_j|<\delta$ for all
$j\ge J$, since $C$ is bounded. Hence $\{B_j\}_{j\ge J}$ is a
$\delta$-packing and $\sum_{j\ge J}h(|B_i|)=+\infty$. Then,
$P_\delta^h(C)=+\infty$ for all $\delta>0$ and therefore,
$P_0^h(C)=+\infty$.

Now suppose $P_0^h(C)=+\infty$. The packing is constructed as
follows. We begin with a $1$-packing $\{B_i^1\}_{i=1}^{N_1}$ of $C$
that satisfies
\begin{itemize}
\item $\sum_{i=1}^{N_1} h(|B_i^1|)>2h(1)$;
\item $P_0^h(C\cap B_{N_1}^1)=+\infty$;
\item dist$(B_i^1, B_{N_1}^1)>0$, \ \ $1\le i<N_1$ (this holds because $h$ is left continuous).
\end{itemize}
As a consequence of these conditions, we can select a $1$-packing
$\{B_i^2\}_{i=1}^{N_2}$ of $C\cap B_{N_1}^1$ such that
\begin{itemize}
\item $\sum_{i=1}^{N_2} h(|B_i^2|)>2h(1)$;
\item $P_0^h(C\cap B_{N_2}^2)=+\infty$;
\item dist$(B_i^2, B_{N_2}^2)>0$ and dist$(B_j^1, B_{N_2}^2)>0$, \ \ $1\le i<N_2$, $1\le j< N_1$;
\item $B_i^2\cap B_j^1=\emptyset$, for all $1\le i\le N_2$, $1\le j< N_1$.
\end{itemize}
Then, $\{B_1^1,\ldots,B_{N_1-1}^1,B_1^2,\ldots,B_{N_2}^2\}$ is a
$1$-packing of $C$ of size
\begin{equation*}
\sum_{i=1}^{N_1-1}h(|B_i^1|)+\sum_{i=1}^{N_2}h(|B_i^2|)>3h(1).
\end{equation*}
Continuing with this procedure \emph{ad infinitum} we obtain the
desired packing.
\end{proof}

\begin{proof}[Proof of Theorem \ref{prepac}\eqref{it:prepac-up}]
By Lemma \ref{propopacinfinite}, there is a $1$-packing
$\{B_i\}$ of $A$ such that $$\sum_i h(|B_i|)=+\infty.$$ Let $N_1,
N_2, \ldots$ be an increasing sequence of integers such that
\begin{equation*}
\sum_{i\le N_j}h(|B_i|)>4^j \quad \text{for  all } j\ge1.
\end{equation*}
Let $t_j\searrow0$ be a sequence of reals such that
$t_j<\min\bigl\{|B_1|,\ldots, |B_{N_j}|\bigr\}$ for all $j\ge1$. We
define $g:[0,+\infty)\to[0,+\infty)$ by
\begin{equation*}
         g(t)=
        \left\{
  \begin{array}{ccc}
     2^{\frac{t-t_j}{t_{j-1}-t_j}-j}h(t), & t\in(t_j, t_{j-1}], \ j>1 \\
          h(t), & t>t_1
  \end{array}
    \right..
\end{equation*}
It is easily seen that $g\in\mathcal D$. Also, observe that for any
$t\in(t_j,t_{j-1}]$,
\begin{equation*}
\frac{g(t)}{h(t)}\le\frac{1}{2^{j-1}},
\end{equation*}
whence $h\prec g$.

Now let $j>1$. Note that if $t_j<t$ for some $j$, then $t\in(t_k,
t_{k-1}]$ for some $k\le j$, and hence
\begin{equation*}\label{eq}
g(t)\ge \frac{h(t)}{2^k}\ge\frac{h(t)}{2^j}.
\end{equation*}
(If $k=1$ then $t\in(t_1,+\infty)$.) Therefore, we obtain
\begin{align*}
\sum_i g(|B_i|)&\ge\sum_{i:t_j<|B_i|}g(|B_i|) \\
               &\ge\sum_{i:t_j<|B_i|}\frac{h(|B_i|)}{2^j} \\
               &\ge\frac{1}{2^j}\sum_{i\le N_j}h(|B_i|) \\
               & > 2^j,
\end{align*}
hence $\{B_i\}$ is a $1$-packing of $A$ such that $\sum_i
g(|B_i|)=+\infty$, and the theorem follows by Proposition
\ref{propopacinfinite}.
\end{proof}

\section{Remarks and a question} \label{sec:remarks}

We finish the article with some remarks and an open question.

\begin{rem}
In Theorem \ref{paczero}, we showed that for a large class of dimension fuctions $h$, there exist sets $E$ with the property that the poset $\{ g\in \mathcal D: \pac^g(E)=0\}$ has $h$ as a minimal element. We remark that, for general sets $E$, the poset may have no minimal elements. For example, if $K$ is the Cantor set constructed in the proof of Theorem \ref{paczero} (for the given dimension function $h$), it is not hard to check that $\pac^g(K)=0$ if and only if $h\prec g$, and this class clearly has no minimal elements (see \cite{CHM10} for the proof of the equivalence when $d=1$).
\end{rem}

\begin{rem}
Theorem \ref{pacinfty} can be generalized to a complete separable metric space but considering the radius-based definition of packing measure instead of the diameter-based  definition (of course in $\rr^d$ both definitions agree, but not in general metric spaces). Indeed, the method of Haase that we adapt works in that general setting.
\end{rem}

In Theorem \ref{paczero}, the sets we construct are $G_\delta$ but not closed. This suggests the following question:

{\bf Open question.} Does there exist a closed set $E\subset\rr^d$ such that the poset $\{ h\in\mathcal D: \pac^h(E)=0 \}$ has a minimal element? If so, what are the possible minimal elements?


\end{document}